\documentclass[12pt]{amsart}\usepackage{amsmath,amsthm,amssymb,cite}

\DeclareMathOperator{\RE}{Re}

\numberwithin{equation}{section}
\newtheorem{theorem}{Theorem}[section]
\newtheorem{lemma}[theorem]{Lemma}

\theoremstyle{remark}
\newtheorem{remark}[theorem]{Remark}

\newtheorem{problem}[theorem]{Problem}

\makeatletter
\@namedef{subjclassname@2010}{%
  \textup{2010} Mathematics Subject Classification}
\makeatother
\allowdisplaybreaks
 
\begin{document}
\author[SUBZAR BEIG]{SUBZAR BEIG}

\address{Department of Mathematics, University of Delhi,
Delhi--110 007, India}
\email{beighsubzar@gmail.com}
\author[V. Ravichandran]{V. Ravichandran}

\address{Department of Mathematics, University of Delhi,
Delhi--110 007, India}
\email{vravi68@gmail.com}

\title[Directional Convexity of Harmonic Mappings ]{directional convexity of harmonic mappings\boldmath}

\begin{abstract}The convolution properties are discussed for the complex-valued harmonic functions in the unit disk $\mathbb{D}$ constructed  from the harmonic shearing of the analytic function $\phi(z):=\int_0^z (1/(1-2\xi\textit{e}^{\textit{i}\mu}\cos\nu+\xi^2\textit{e}^{2\textit{i}\mu}))\textit{d}\xi$, where $\mu$ and $\nu$ are real numbers. For any real number $\alpha$ and harmonic function $f=h+\overline{g}$, define  an analytic function $f_{\alpha}:=h+\textit{e}^{-2\textit{i}\alpha}g$. Let $\mu_1$ and $\mu_2$ $(\mu_1+\mu_2=\mu)$ be real numbers, and  $f=h+\overline{g}$ and $F=H+\overline{G}$ be  locally-univalent and sense-preserving harmonic functions such that $f_{\mu_1}*F_{\mu_2}=\phi$. It is shown that the convolution $f*F$ is univalent and convex in the direction of $-\mu$, provided it is locally univalent and sense-preserving. Also, local-univalence of the above convolution $f*F$ is shown for some specific analytic dilatations of $f$ and $F$. Furthermore, if $g\equiv0$ and both the analytic functions $f_{\mu_1}$ and $F_{\mu_2}$ are  convex, then the convolution $f*F$ is shown to be convex. These results extends the work done by Dorff \textit{et al.} to a larger class of functions.

\end{abstract}


\thanks{The first author is supported by  a Junior Research Fellowship from University Grants Commission, New Delhi, India.}
\maketitle

\section{Introduction}
Let $\mathcal{H}$ be the class of all complex-valued harmonic functions $f$ defined on the unit disk $\mathbb{D}=\left\lbrace z\in\mathbb{C}: |z|<1\right\rbrace$. Such functions can be expressed as $f=h+\overline{g}$, where $h$ and $g$ are analytic functions on $\mathbb{D}$ and are respectively known as analytic and co-analytic parts of $f$. We consider the functions to be the normalized one, that is functions $f$ in  $\mathcal{H}$ satisfy the conditions, $h(0)=h_z(0)-1=g(0)=0$. Let $\mathcal{S}_H$ be the sub-class of $\mathcal{H}$ consisting of all univalent harmonic functions, and let $\mathcal{S}_H^0:=\left\lbrace f\in\mathcal{S}_H:f_{\bar{z}}(0)=0\right\rbrace$. The sub-classes $\mathcal{K}_H$, $\mathcal{S}^{*}_H$ and $\mathcal{C}_H$ of $\mathcal{S}_H$ (resp. $\mathcal{K}_H^0$, $\mathcal{S}^{*0}_H$ and $\mathcal{C}_H^0$ of $\mathcal{S}_H^0$), which maps $\mathbb{D}$ respectively onto convex, starlike and close-to-convex domains, were studied by Cluine and Sheil-Small in \cite{cluine}. With the co-analytic part $g\equiv 0$, the class $\mathcal{S}_H^0$ reduces to  $\mathcal{S}$, the class of all  normalized analytic univalent mappings in $\mathbb{D}$. The classes $\mathcal{K}$, $\mathcal{S}^*$ and $\mathcal{C}$ (respectively known as convex, starlike and close-to-convex analytic functions) are respectively the sub-classes of $\mathcal{K}_H^0$, $\mathcal{S}^{*0}_H$ and $\mathcal{C}_H^0$, consisting of all functions $f=h+\overline{g}$ with $g\equiv 0$. One of the important fields in the geometric function theory is the study of the convolution (or Hardmard product) of functions. Let the functions $\phi_1$ and $\phi_1$ be analytic in $\mathbb{D}$, with the Taylor series expansion as:\[ \phi_1(z)=\sum_{n=0}^{\infty}a_n z^n\quad\text{ and}\quad \phi_2(z)=\sum_{n=0}^{\infty}b_n z^n.\]Then the convolution $*$ of $\phi_1$ and $\phi_1$ is defined as: \[(\phi_1*\phi_2)((z):=\sum_{n=0}^{\infty}a_nb_n z^n.\]Also, the convolution $*$ of two harmonic functions $f_1=h_1+\overline{g_1}$ and $f_2=h_2+\overline{g_2}$ is defined as: $f_1*f_2:=h_1*h_1+\overline{g_1*g_2}$, and the convolution $\tilde{*}$ of the  analytic function $\phi$ with the harmonic function $f=h+\overline{g}$ is defined as: $f\tilde{*}\phi:=h*\phi+\overline{g*\phi}$.\

In \cite{rusch}, Ruscheweyh and Sheil-Small showed that the class $\mathcal{K}$ is closed under convolution. That is, if  the functions  $\phi_1$, $\phi_1\in\mathcal{K}$, then the function $\phi_1*\phi_2\in\mathcal{K}$. However, such a result is not true for the corresponding class $\mathcal{K}_H^0$ of harmonic functions. In the case of harmonic mappings, the convolution need not be univalent. In this direction Cluine and Sheil-Small proposed the problem, known as \textit{multiplier problem}:  if the function $f\in\mathcal{K}_H^0$, then what are the functions $g\in\mathcal{K}_H^0$ such that  $f*g\in\mathcal{K}_H^0$? This problem was partially solved by Ruscheweyh and Salins in \cite{rushsalin}. In Section 3, we prove that the convolution of some analytic convex functions with non-convex  harmonic functions belongs to $\mathcal{K}_H^0$. In particular, it is shown that the convolution of the analytic function $\int_0^z \left(1/(1-2\xi\textit{e}^{\textit{i}\mu}\cos\nu+\xi^2\textit{e}^{2\textit{i}\mu})\right)\textit{d}\xi$ function with the locally univalent and sense-preserving harmonic function $f=h+\overline{g}$, satisfying $h(z)+\textit{e}^{-2\textit{i}\mu}g(z)=z/(1-z)$, is convex. Another important class, which is of our interest as well in this paper, is the class of univalent functions convex in a particular direction. A domain $\mathcal{D}$ is said to be convex in the direction of $\gamma$ $(0\leq\gamma<\pi)$, if every line parallel to the line joining origin to the point $\textit{e}^{\textit{i}\gamma}$ has connected intersection with $\mathcal{D}$. If $\gamma=0$ (or $\pi/2)$; such a domain is said to be convex in the direction of real (or imaginary) axis. A function is said to be convex in some direction, if it maps $\mathbb{D}$ to a domain which is convex in that particular direction. Such functions are close to convex. Functions convex in every direction are convex functions. In \cite{cluine}, Cluine and Sheil-Small gave a result, which gives a method known as  method of \textit{Shear Construction}, to check the convexity in a particular direction or convexity of harmonic functions. In particular, they gave the following result.
\begin{lemma}\label{p3lema01}\cite{cluine}
A locally univalent and sense-preserving harmonic function $f=h+\overline{g}$ on $\mathbb{D}$ is univalent and maps $\mathbb{D}$ onto a domain convex in the direction of $\phi$ if and only if the analytic mapping $h-\emph{e}^{2\textit{i}\phi}g$ is univalent and maps $\mathbb{D}$ onto a  domain convex in the direction of $\phi$.
\end{lemma}
A function $f_1=h_1+\overline{g_1}\in\mathcal{K}_H^0$ is known as a right half-plane mapping, if it maps $\mathbb{D}$ onto the right half-plane $R=\left\lbrace w\in\mathbb{C}:\RE(w)>-1/2\right\rbrace$. For $\pi/2<\mu_2<\pi$, a function   $f_2=h_2+\overline{g_2}\in\mathcal{K}_H^0$ that maps $\mathbb{D}$ onto the vertical strip \[ \Omega_{\mu}=\left\lbrace w\in\mathbb{C}:\frac{\mu-\pi}{2\sin\mu}<\RE(w)<\frac{\mu}{2\sin\mu}\right\rbrace\] is known as a vertical strip mapping, and the class of all such mappings is denoted by $\mathcal{S}^0(\Omega_{\mu})$. Such mappings; $f_1$ and$f_2$ satisfy the following, (see \cite{abu,dorffvstrip})
\begin{equation}\label{p3eq02}
h_1(z)+g_1(z)=\frac{z}{1-z},
\end{equation}and
\begin{equation}\label{p3eq03}
h_2(z)+g_2(z)=\frac{1}{2\textit{i}\sin\mu}\log\left(\frac{1+z\textit{e}^{\textit{i}\mu}}{1+z\textit{e}^{-\textit{i}\mu}}\right).
\end{equation}If $\omega(z)=-z$ is the analytic dilatation of the function $f_1=h_1+\overline{g_1}$, that is $g_1'(z)=-zh_1'(z)$, then, in view of (1.1), we get\begin{equation}\label{p3eq04}
h_1(z)=\frac{1}{2}\left(\frac{z}{1-z}+\frac{z}{(1-z)^2}\right)\quad \text{and} \quad g_1(z)=\frac{1}{2}\left(\frac{z}{1-z}-\frac{z}{(1-z)^2}\right).
\end{equation}
This right half-plane mapping acts as extremal function for many problems for the class of convex harmonic functions, (see for example \cite{cluine}).
Also, in \cite{dnowak}, it is shown that, for $\alpha$ real, a function $f=h+\overline{g}\in\mathcal{K}_H^0$  that maps $\mathbb{D}$ onto the slanted half-plane $H_{\alpha}=\left\lbrace w\in\mathbb{C}:\RE(\textit{e}^{\textit{i}\alpha}w)>-1/2\right\rbrace$ satisfy
\begin{equation}\label{p3eq05}
h(z)+\textit{e}^{-2\textit{i}\alpha}g(z)=\frac{z}{1-\textit{e}^{\textit{i}\alpha}z}.
\end{equation} Such a mapping is known as a slanted half-plane mapping, and the class of all such mappings is denoted by $\mathcal{S}^0(H_{\alpha})$. If $\alpha=0$, the above mapping is a right half-plane mapping. We prove a similar result for the strip mappings. For $\pi/2<\mu<\pi$ and real number $\alpha$, if a function $f=h+\overline{g}\in\mathcal{K}_H^0$ maps $\mathbb{D}$ onto the slanted strip \[ \Omega_{\mu,\alpha}=\left\lbrace w\in\mathbb{C}:\frac{\mu-\pi}{2\sin\mu}<\RE(\textit{e}^{\textit{i}\alpha}w)<\frac{\mu}{2\sin\mu}\right\rbrace,\] we call it as a slanted strip mapping and denote the class of all such mappings by $\mathcal{S}^0(\Omega_{\mu,\alpha})$. Clearly $\mathcal{S}^0(\Omega_{\mu,0})=\mathcal{S}^0(\Omega_{\mu})$. Following result gives an explicit description of such mappings.
\begin{lemma}If the function $f=h+\overline{g}\in\mathcal{S}^0(\Omega_{\mu,\alpha})$, then\[h(z)+\textit{e}^{-2\textit{i}\alpha}g(z)=\frac{\textit{e}^{-\textit{i}\alpha}}{2\textit{i}\sin\mu}\log\left(\frac{1+z\textit{e}^{\textit{i}(\alpha+\mu)}}{1+z\textit{e}^{\textit{i}(\alpha-\mu)}}\right).\]
\end{lemma}
\begin{proof}
Let the function $f=h+\overline{g}\in\mathcal{S}^0(\Omega_{\mu,\alpha})$. Then, the function $z\rightarrow\textit{e}^{\textit{i}\alpha}f(z)$ maps $\mathbb{D}$ onto the vertical strip $\Omega_{\mu}$, and hence the function \begin{equation}\label{p3eq05a}H(w)+\overline{G(w)}:=\textit{e}^{\textit{i}\alpha}h(\textit{e}^{-\textit{i}\alpha}w)+\overline{\textit{e}^{-\textit{i}\alpha}g(\textit{e}^{-\textit{i}\alpha}w)},
\end{equation} where $w=z\textit{e}^{\textit{i}\alpha}$, maps $\mathbb{D}$ onto the vertical strip $\Omega_{\mu}$. Also, normalization of $f$ gives that $H(0)=H_w(0)-1=G(0)=G_w(0)=0$. Therefore, the function $H+\overline{G}\in\mathcal{S}^0(\Omega_{\mu})$. Hence, \eqref{p3eq03} gives
\[H(w)+G(w)=\frac{1}{2\textit{i}\sin\mu}\log\left(\frac{1+w\textit{e}^{\textit{i}\mu}}{1+w\textit{e}^{-\textit{i}\mu}}\right).\]
Substituting the values of $H$ and $G$ from \eqref{p3eq05a} and replacing $w$ by $z\textit{e}^{\textit{i}\alpha}$ in  the above equation, we get the desired result.
\end{proof}
 Using Lemma \ref{p3lema01}, Dorff \cite{dorff} studied the directional convexity of the convolution of right half-plane and vertical strip mappings. Later on, Dorff \textit{et al.}\cite{dnowak} extended such study to slanted  half-plane mappings as well. In these papers the problem of directional convexity of the convolution of such functions is actually reduced to the local univalence and sense-preservity of the convolution function. In fact, they proved the following results.
\begin{lemma}\label{p3lema06}\cite{dnowak}
Let the function $f_k\in\mathcal{S}^0(H_{\gamma_k})$, $(k=1,2)$. Then the function $f_1*f_2\in\mathcal{S}_H^0$ and  is convex in the direction of $-(\gamma_1+\gamma_2)$, if it is locally univalent and sense-preserving in $\mathbb{D}$.
\end{lemma}
\begin{lemma}\label{p3lema07}\cite{dnowak}
Let the function  $f_1=h_1+\overline{g_1}$ be a right half-plane mapping  and  the function $f_2=h_2+\overline{g_2}$ be a strip mapping as defined above. Then the function $f_1*f_2\in\mathcal{S}_H^0$ and is convex in the direction of the real axis, if it is locally univalent and sense-preserving in $\mathbb{D}$.
\end{lemma}
Furthermore, after fixing the function $f_1$ to be the right half-plane mapping defined in \eqref{p3eq04}, they proved the local univalence of the convolution function $f_1*f_2$ for some special analytic dilatations of the function $f_2$. In fact by using the above two lemmas, they proved the following results.
\begin{theorem}\label{p3lema08}\cite{dnowak}
Let the function $f_1$ be the right half-plane mapping given by \eqref{p3eq04} and the function $f_2=h_2+\overline{g_2}\in\mathcal{K}_H^0$ be a slanted half-plane mapping. If $\omega(z)=\textit{e}^{\textit{i}\theta}z^n$ is the  analytic dilation of $f_2$, then, for $\theta\in\mathbb{R}$ and $n=1,2$, the convolution $f_1*f_2\in\mathcal{S}_H^0$ and  is convex in the direction of real axis.
\end{theorem}
\begin{theorem}\label{p3lema09}\cite{dnowak}
Let the function $f_1$ be the right half-plane mapping given by \eqref{p3eq04} and the function $f_2=h_2+\overline{g_2}\in\mathcal{S}^0(\Omega_{\alpha})$ be a vertical strip  mapping. If $\omega(z)=\textit{e}^{\textit{i}\theta}z^n$ is the  analytic dilation of $f_2$, then, for $\theta\in\mathbb{R}$ and $n=1,2$, the convolution $f_1*f_2\in\mathcal{S}_H^0$ and  is convex in the direction of real axis.
\end{theorem}
 Later on, in \cite{lipona} Li and Ponnusamy  improved the above two results and proved the following results.
\begin{theorem}\label{p3lema010}\cite{lipona}
Let the function $f_1$ be the right half-plane mapping given by \eqref{p3eq04}. Also, let the function $f_2=h_2+\bar{g_2}\in\mathcal{S}^0(H_{\alpha})$ be a slanted half-plane mapping and  $\omega=az^n$, $(a\in\mathbb{C},n\in\mathbb{N})$ be its analytic dilation. Then the convolution $f_1*f_2$ is univalent and convex in the direction of $-\alpha$, if
\begin{enumerate}
\item  $n=1,2$ and $|a|\leq 1$, or
\item  $n\geq3$ and $|a|\leq n-1-\sqrt{n^2-2n}$.
\end{enumerate}
\end{theorem}
\begin{theorem}\label{p3lema011}\cite{lipona}
Let the function $f_1$ be the right half-plane mapping given by \eqref{p3eq04}. Also, let function $f_2=h_2+\bar{g_2}\in\mathcal{S}^0(\Omega_{\alpha})$ be a vertical strip mapping and  $\omega=az^n$, $(a\in\mathbb{C},n\in\mathbb{N})$ be its analytic dilation. Then the convolution $f_1*f_2$ is univalent and convex in the direction of real axis, if
\begin{enumerate}
\item  $n=1,2$ and $|a|\leq 1$, or
\item  $n\geq3$ and $|a|\leq n-1-\sqrt{n^2-2n}$.
\end{enumerate}
\end{theorem}
In this direction, we find out that the results in Lemma \ref{p3lema06} and Lemma \ref{p3lema07} depend upon the convolution of functions in the right-hand sides of \eqref{p3eq02}, \eqref{p3eq03}  and \eqref{p3eq05}. In fact, we find out that such results work for a larger class of functions, which can be determined by taking the  harmonic shears of the convex functions which upon convolution gives the function $\int_0^z (1/(1-2\xi\textit{e}^{\textit{i}\mu}\cos\nu+\xi^2\textit{e}^{2\textit{i}\mu}))\textit{d}\xi$, (see Theorem 2.3). Also, in last theorem, we investigate the local univalence of the convolution of such functions for some choices of analytic dilatations of these functions. In this theorem, not only we consider a larger class of functions than those considered in the above results, but we also vary the function $f_1$ which is taken to be fixed right half-plane mapping in the above results.

\section{Main Results}
We  will begin this section with the following theorem, which will be useful in finding out the local univalence of the convolution of harmonic functions.
\begin{theorem}\label{p3theom1}
Let the harmonic functions $f_1=h_1+\overline{g_1}$ and $f_2=h_2+\overline{g_2}$ be locally univalent and sense-preserving in $\mathbb{D}$ such that, for some real numbers $\mu_1$ and $\mu_2$, the functions $h_1+\textit{e}^{-2\textit{i}\mu_1}g_1$, $h_2+\textit{e}^{-2\textit{i}\mu_2}g_2\in\mathcal{K}$. Also, let any analytic function $F$ satisfying\[\RE\left(\frac{zF'(z)}{G(z)}\right)\geq0,\quad z\in\mathbb{D},\] where $G=z((h_1+\textit{e}^{-2\textit{i}\mu_1}g_1)*(h_2+\textit{e}^{-2\textit{i}\mu_2}g_2))'$, implies that $F$ is convex in the direction of $-(\mu_1+\mu_2)$. Then the convolution $f_1*f_2$ is univalent and convex in the direction of $-(\mu_1+\mu_2)$, if it is locally univalent and sense-preserving.
\end{theorem}
\begin{proof}
Consider  the functions $F_1$ and $F_2$ defined by\[F_1=:\left(h_1+\textit{e}^{-2\textit{i}\mu_1}g_1\right)*\left(h_2-\textit{e}^{-2\textit{i}\mu_2}g_2\right)\quad\text{and}\quad F_2=:\left(h_1-\textit{e}^{-2\textit{i}\mu_1}g_1\right)*\left(h_2+\textit{e}^{-2\textit{i}\mu_2}g_2\right).\]A calculation shows that
\begin{align}
\RE\left(\frac{zF_1'(z)}{G(z)}\right)&
=\RE\left(\frac{\left(h_1+\textit{e}^{-2\textit{i}\mu_1}g_1\right)(z)
*z\left(h_2-\textit{e}^{-2\textit{i}\mu_2}g_2\right)'(z)} {G(z)}\right)\notag\\
&=\RE\left(\frac{\left(h_1+\textit{e}^{-2\textit{i}\mu_1}g_1\right)(z)
*\frac{(h_2-\textit{e}^{-2\textit{i}\mu_2}g_2)'(z)}{(h_2+\textit{e}^{-2\textit{i}\mu_2}g_2)'(z)}
z\left(h_2+\textit{e}^{-2\textit{i}\mu_2}g_2\right)'(z)}
{z((h_1+\textit{e}^{-2\textit{i}\mu_1}g_1)(z)*(h_2+\textit{e}^{-2\textit{i}\mu_2}g_2))'(z)}\right)\notag\\
&=\RE\left(\frac{\left(h_1+\textit{e}^{-2\textit{i}\mu_1}g_1\right)(z)*P_2(z)z\left(h_2
+\textit{e}^{-2\textit{i}\mu_2}g_2\right)'(z)}{(h_1+\textit{e}^{-2\textit{i}\mu_1}g_1)(z)
*z(h_2+\textit{e}^{-2\textit{i}\mu_2}g_2)'(z)}\right),\label{p3eq2}
\end{align}
where $$P_2(z)=\frac{(h_2-\textit{e}^{-2\textit{i}\mu_2}g_1)'(z) }{(h_2+\textit{e}^{-2\textit{i}\mu_2}g_2)'(z)}.$$ Since the function $f_2=h_2+\overline{g_2}$ is locally univalent and sense-preserving, its dilatation $\omega_2=g_2'/h_2'$ satisfies $|\omega_2(z)|<1$ for $z\in\mathbb{D}$. Hence, $\RE (P_2(z))>0 $ for $z\in\mathbb{D}$. Also, the function	 $(h_1+\textit{e}^{-2\textit{i}\mu_1}g_1)\in\mathcal{K}$ and the function $z(h_2+\textit{e}^{-2\textit{i}\mu_2}g_2)'\in\mathcal{S}^*$. Therefore, in view of \eqref{p3eq2}, a result in \cite{rusch} gives
\begin{equation}\label{p3eq3}
\RE\left(\frac{zF_1'(z)}{G(z)}\right)>0,\quad z\in\mathbb{D}.
\end{equation}
Similarly we will get
\begin{equation}\label{p3eq4}
\RE\left(\frac{zF_2'(z)}{G(z)}\right)>0,\quad z\in\mathbb{D}.
\end{equation}
In view of \eqref{p3eq3} and  \eqref{p3eq4}, the function $F$ defined by \[ F :=\frac{1}{2}\left(F_1+F_2\right)=h_1*h_2-\textit{e}^{-2\textit{i}(\mu_1+\mu_2)}g_1*g_2
\] satisfies
\begin{equation}\label{p3eq5}
\RE\left(\frac{zF'(z)}{G(z)}\right)>0,\quad z\in\mathbb{D}.
\end{equation}
Therefore, by assumption in the statement of the theorem, the function $F=h_1*h_2-\textit{e}^{-2\textit{i}(\mu_1+\mu_2)}g_1*g_2$ is univalent and convex in the direction of $-(\mu_1+\mu_2)$. The result now follows by invoking Lemma \ref{p3lema01}.
\end{proof}
Now we recall a result of Royster and Zeigler \cite{royster} for checking the directional convexity of analytic functions.
\begin{theorem}\label{p3theom6}\cite{royster}
Let $\phi$ be a non-constant analytic function in $\mathbb{D}$. Then $\phi$ maps $\mathbb{D}$ onto a domain convex in the direction of imaginary axis if and only if there are real numbers $\mu$ $(0\leq\mu<2\pi)$ and $\nu$ $(0\leq\nu<\pi)$, such that
\begin{equation}\label{p3eq7}
\RE\left\lbrace-\textit{i}\textit{e}^{\textit{i}\mu}(1-2z\textit{e}^{-\textit{i}\mu}\cos\nu+z^2\textit{e}^{-2\textit{i}\mu})\phi'(z) \right\rbrace\geq0,\quad z\in \mathbb{D}.
\end{equation}
\end{theorem}
Since a function  $\phi$ is convex in the direction $\gamma$ if and only if the function $\textit{e}^{\textit{i}(\pi/2-\gamma)}\phi$ is convex in the direction of imaginary axis, Theorem \ref{p3theom6} gives the following criteria for a function to be convex in the direction of $\gamma$.
\begin{theorem}\label{p3theom8}
Let $\phi$ be a non-constant analytic function in $\mathbb{D}$. Then $\phi$ maps $\mathbb{D}$ onto a domain convex in the direction of $\gamma$ if and only if there are real numbers $\mu$ $(0\leq\mu<2\pi)$ and $\nu$ $(0\leq\nu<\pi)$, such that
\begin{equation}\label{p3eq9}
\RE\left\lbrace\textit{e}^{\textit{i}(\mu-\gamma)}(1-2z\textit{e}^{-\textit{i}\mu}\cos\nu+z^2\textit{e}^{-2\textit{i}\mu})\phi'(z) \right\rbrace\geq0,\quad z\in \mathbb{D}.
\end{equation}
\end{theorem}
Using Theorem \ref{p3theom1} and Theorem \ref{p3theom8}, we get the following result.
\begin{theorem}\label{p3theom10}
Let the functions $f_1=h_1+\overline{g_1}$ and $f_2=h_2+\overline{g_2}$ be locally univalent and sense-preserving harmonic mappings in $\mathbb{D}$ such that for some real numbers $\mu_1$ and $\mu_2$, the functions $h_1+\textit{e}^{-2\textit{i}\mu_1}g_1$, $h_2+\textit{e}^{-2\textit{i}\mu_2}g_2\in\mathcal{K}$ and
\begin{equation}\label{p3eq11}
 (h_1+\textit{e}^{-2\textit{i}\mu_1}g_1)(z)*(h_2+\textit{e}^{-2\textit{i}\mu_2}g_2)(z)=\int_0^z \frac{\textit{d}\xi}{1-2\xi\textit{e}^{\textit{i}(\mu_1+\mu_2)}\cos\nu+\xi^2\textit{e}^{2\textit{i}(\mu_1+\mu_2)}}.
\end{equation}
Then the convolution $f_1*f_2$ is univalent and convex in the direction of $-(\mu_1+\mu_2)$, if it is locally univalent and sense-preserving.
\end{theorem}
\label{p3remak11a} Let the function $f_k=h_k+\overline{g_k}\in\mathcal{S}^0(H_{\mu_k})$, $(k=1,2)$. Then, by \eqref{p3eq05}, the functions $h_1+\textit{e}^{-2\textit{i}\mu_1}g_1$, $h_2+\textit{e}^{-2\textit{i}\mu_2}g_2\in\mathcal{K}$ and satisfy
\begin{align*}
 (h_1+\textit{e}^{-2\textit{i}\mu_1}g_1)(z)*(h_2+\textit{e}^{-2\textit{i}\mu_2}g_2)(z)&=\frac{z}{1-\textit{e}^{\textit{i}(\mu_1+\mu_2)}z}\\&=\int_0^z \frac{\textit{d}\xi}{1-2\xi\textit{e}^{\textit{i}(\mu_1+\mu_2)}+\xi^2\textit{e}^{2\textit{i}(\mu_1+\mu_2)}},
\end{align*}which is equivalent to \eqref{p3eq11} with $\nu=0$. Therefore, by Theorem \ref{p3theom10}, the convolution $f_1*f_2$ is univalent and convex in the direction of $-(\mu_1+\mu_2)$, if it is locally univalent and sense-preserving.
\label{p3remak11b} Also, let the function  $f_1=h_1+\overline{g_1}$  be a right half-plane mapping and  the function $f_2=h_2+\overline{g_2}\in\mathcal{S}^0(\Omega_{\nu})$ be a strip mapping. Then, by \eqref{p3eq02} and \eqref{p3eq03}, the functions $h_1+g_1$, $h_2+g_2\in\mathcal{K}$ and satisfy
\begin{align*}
 (h_1+g_1)(z)*(h_2+g_2)(z)&=\frac{1}{2\textit{i}\sin\nu}\log\left(\frac{1+z\textit{e}^{\textit{i}\nu}}{1+z\textit{e}^{-\textit{i}\nu}}\right)\\&=\int_0^z \frac{\textit{d}\xi}{1-2\xi\cos\nu+\xi^2}.
\end{align*}which is equivalent to \eqref{p3eq11} with $\mu_1=\mu_2=0$. Therefore, by Theorem \ref{p3theom10}, the convolution $f_1*f_2$ is univalent and convex in the direction of real axis, if it is locally univalent and sense-preserving. The above discussion shows that the Theorem \ref{p3theom10} is a generalization of both Lemma \ref{p3lema06} and Lemma \ref{p3lema07}.\

The problem of interest now is to find out whether the solutions of \eqref{p3eq11} exists or not. In other words, whether right hand side of \eqref{p3eq11} can be written as convolution of two functions in the class $\mathcal{K}$ or not. We will show that such solutions exists. Since the function $z/(1-z)\in\mathcal{K}$ is convolution identity, for every function in $\mathcal{K}$ such solutions exists. So in order to prove that the solutions exist for \eqref{p3eq11}, it is enough to prove that the right hand side of \eqref{p3eq11}  is in $\mathcal{K}$\

Let, for some real numbers $\mu$ and $\nu$ satisfying $0\leq\mu<2\pi$ and $0\leq\nu<\pi$,
\begin{equation}\label{p3eq12}
\phi(z)=\int_0^z \frac{\textit{d}\xi}{1-2\xi\textit{e}^{\textit{i}\mu}\cos\nu+\xi^2\textit{e}^{2\textit{i}\mu}}=\int_0^z \frac{\textit{d}\xi}{(1-\xi\textit{e}^{\textit{i}(\mu+\nu)})(1-\xi\textit{e}^{\textit{i}(\mu-\nu)})}.
\end{equation}
Clearly the function $\phi$ is analytic on $\mathbb{D}$. Now, on differentiating \eqref{p3eq12}, we get
\begin{equation}\label{p3eq13}
\phi'(z)= \frac{1}{1-2z\textit{e}^{\textit{i}\mu}\cos\nu+z^2\textit{e}^{2\textit{i}\mu}}.
\end{equation}
Again, on differentiating \eqref{p3eq13}, we get
\begin{equation}\label{p3eq14}
\phi''(z)= \frac{2\textit{e}^{\textit{i}\mu}\cos\nu-2z\textit{e}^{2\textit{i}\mu}}{(1-2z\textit{e}^{\textit{i}\mu}\cos\nu+z^2\textit{e}^{2\textit{i}\mu})^2}.
\end{equation}
Using \eqref{p3eq13} and \eqref{p3eq14}, we see for $z\in\mathbb{D}$,
\begin{align*}
\RE\left(1+z\frac{\phi''(z)}{\phi'(z)}\right)&=\RE\left(\frac{1-z^2\textit{e}^{2\textit{i}\mu}}{1-2z\textit{e}^{\textit{i}\mu}\cos\nu+z^2\textit{e}^{2\textit{i}\mu}}\right)\\&=\frac{1-|z|^4-2\cos\nu(1-|z|^2)\RE(\textit{e}^{\textit{i}\mu}z)}{|1-2z\textit{e}^{\textit{i}\mu}\cos\nu+z^2\textit{e}^{2\textit{i}\mu}|^2}\\&\geq\frac{(1-|z|^2)(1+|z|^2-2|\cos\nu|\RE(\textit{e}^{\textit{i}\mu}z))}{|1-2z\textit{e}^{\textit{i}\mu}\cos\nu+z^2\textit{e}^{2\textit{i}\mu}|^2}>0.
\end{align*}
Also, $\phi(0)=0$ and $\phi'(0)=1$. Therefore, the function $\phi\in\mathcal{K}$.\
\section{Convolution of convex mappings}
 Let $\mathcal{K}(\varphi)$ be the set of  analytic functions in $\mathbb{D}$ which are convex in the direction of $\varphi$. The set $\textit{DCP}$ represents all  analytic functions $g$ in $\mathbb{D}$ such that the convolution $g*f \in \mathcal{K}(\varphi)$
for every $\varphi\in\mathbb{R}$ and every $f \in \mathcal{K}(\varphi)$. Ruscheweyh and Salins \cite{rushsalin} gave a partial proof of the \textit{multiplier problem}, given in the introduction, and proved the following.
\begin{theorem}\label{p3theom14a}\cite{rushsalin}
Let the function $\phi$ be analytic in $\mathbb{D}$. Then the convolution $f\tilde{*}\phi\in\mathcal{K}_H$ for all the functions $f\in\mathcal{K}_H$ if and only if the function $\phi\in\textit{DCP}$.
\end{theorem}
Above result talks of convex harmonic functions. Next two results provides us examples in which we see convolution of some analytic convex functions with some non-convex harmonic functions is convex. These functions are actually determined implicitly by a class of analytic convex functions. Here the non-convex harmonic functions $f=h+\bar{g}$ considered satisfy  $(h-\textit{e}^{2\textit{i}\varphi}g)\in\mathcal{K}$ for some real number $\varphi$.
\begin{theorem}\label{p3theom14ab}
Let the function $\phi\in\mathcal{K}$ and the function $f=h+\overline{g}$ be a locally univalent and sense-preserving harmonic function such that for some real number $\varphi$,
\begin{equation}\label{p3eq14b}
h-\textit{e}^{2\textit{i}\varphi}g\in\mathcal{K}.
\end{equation}
 Then the convolution $f\tilde{*}\phi$ is univalent and convex in the direction of $\varphi$. Furthermore, if the function $(h-\textit{e}^{2\textit{i}\alpha}g)*\phi$ is convex in the direction of $\alpha$ for some real number $\alpha$, then the convolution $f\tilde{*}\phi$ is also convex in the direction of $\alpha$.
\begin{proof}
In view of Lemma \ref{p3lema01}, it is enough to show that the function $f\tilde{*}\phi=h*\phi+\overline{g*\phi}$ is locally univalent and sense-preserving, or by Lewy's theorem it reduces to showing that its dilatation $\omega=(g*\phi)'/(h*\phi)'$ satisfies $|\omega|<1$ on $\mathbb{D}$. First, we note that
\begin{align}
\RE\left(\frac{1+\omega}{1-\omega}\right)&=\RE\left(\frac{(h*\phi)'+(g*\phi)'}{(h*\phi)'-(g*\phi)'}\right)\notag\\&=2\RE\left(\frac{(h*\phi)'}{(h*\phi)'-(g*\phi)'}\right)-1\notag\\&=2\RE\left(\frac{\phi*z(h-\textit{e}^{2\textit{i}\varphi}g)'\left(\frac{h'}{h'-\textit{e}^{2\textit{i}\varphi}g'}\right)}{\phi*z(h-\textit{e}^{2\textit{i}\varphi}g)'}\right)-1. \label{p3eq14c}
\end{align}
Since the function $f=h+\bar{g}$ is locally univalent and sense-preserving, therefore $|g'/h'|<1$ on $\mathbb{D}$, or equivalently $\RE\left(h'/(h'-\textit{e}^{2\textit{i}\varphi}g')\right)>1/2$ on $\mathbb{D}$. Also, the function $\phi\in\mathcal{K}$ and the function $z(h-\textit{e}^{2\textit{i}\varphi}g)'\in\mathcal{S}^*$. Therefore, in view of \eqref{p3eq14c}, a result in \cite{rusch} gives $\RE((1+\omega)/(1-\omega))>0$ on $\mathbb{D}$, or equivalently  $|\omega|<1$ on $\mathbb{D}$.
\end{proof}
\end{theorem}
Next result provides examples for Theorem \ref{p3theom14ab}.
\begin{theorem}\label{p3theom14d}
Let the function $\phi\in\mathcal{K}$ and the function $f=h+\overline{g}$  be a locally univalent and sense-preserving harmonic mapping in $\mathbb{D}$ such that for some real numbers $\mu$ and $\nu$, the function $h+\textit{e}^{-2\textit{i}\mu_1}g\in\mathcal{K}$ and
\begin{equation}\label{p3eq14e}
 ((h+\textit{e}^{-2\textit{i}\mu}g)*\phi)(z)=\int_0^z \frac{\textit{d}\xi}{1-2\xi\textit{e}^{\textit{i}\mu}\cos\nu+\xi^2\textit{e}^{2\textit{i}\mu}}.
\end{equation}
Then the convolution $f\tilde{*}\phi\in\mathcal{K}_H$.
\end{theorem}
\begin{proof}
Clearly the functions $f$ and $\phi$ satisfy the assumptions in Theorem \ref{p3theom14a}, and hence the convolution $f*\phi$ is univalent. In order to prove the result, in view of Lemma \ref{p3lema01}, it suffices to show that the function $(h-\textit{e}^{2\textit{i}\theta}g)*\phi$ is convex in the direction $\theta$ for all $\theta$ ranging in an interval of length $\pi$. In other words, it is sufficient to show that the function $\textit{e}^{-\textit{i}(\mu+\theta)}(h-\textit{e}^{2\textit{i}\theta}g)*\phi$ is convex in the direction $-\mu$ for all $\theta$ such that $0\leq\mu+\theta<\pi$. Consider the case $0\leq\mu+\theta<\pi/2$. Since  $f*\phi$ is univalent, its dilation $(g*\phi)'/(h*\phi)'$  lies in $\mathbb{D}$ and hence  $$\RE\left(\frac{(h'-\textit{e}^{-2\textit{i}\mu}g')*\phi}{(h'+\textit{e}^{-2\textit{i}\mu}g')*\phi}\right)>0.$$Using above inequality, we have
\begin{align}
\RE\left(\frac{(\textit{e}^{-\textit{i}(\mu+\theta)}h-\textit{e}^{\textit{i}(\theta-\mu)}g)'*\phi}{(h+\textit{e}^{-2\textit{i}\mu}g)'*\phi}\right)&=\RE\left(\frac{(\textit{e}^{-\textit{i}(\mu+\theta)}h'-\textit{e}^{-2\textit{i}\mu}\textit{e}^{\textit{i}(\mu+\theta)}g')*\phi}{(h'+\textit{e}^{-2\textit{i}\mu}g')*\phi}\right)\notag\\&=\RE\left(\frac{(h'-\textit{e}^{-2\textit{i}\mu}g')*\phi}{(h'+\textit{e}^{-2\textit{i}\mu}g')*\phi}\cos(\mu+\theta)-\textit{i}\sin(\mu+\theta)\right)\notag\\&=\cos(\mu+\theta)\RE\left(\frac{(h'-\textit{e}^{-2\textit{i}\mu}g')*\phi}{(h'+\textit{e}^{-2\textit{i}\mu}g')*\phi}\right)>0.\label{p3eq14f}
\end{align}
Now, \eqref{p3eq14e} gives $((h'-\textit{e}^{-2\textit{i}\mu}g')*\phi)(z)=1/(1-2z\textit{e}^{\textit{i}\mu}\cos\nu+z^2\textit{e}^{2\textit{i}\mu})$. Therefore, in view of \eqref{p3eq14f}, Theorem \ref{p3theom8} after taking $\gamma=\mu$ shows that the function $\textit{e}^{-\textit{i}(\mu+\theta)}(h-\textit{e}^{2\textit{i}\theta}g)*\phi$ is convex in the direction of $-\mu$ for all $\theta$ such that $0\leq\mu+\theta<\pi/2$. Taking $\gamma=\mu+\pi$ in Theorem \ref{p3theom8} and proceeding similarly as above for the case $\pi/2\leq\mu+\theta<\pi$.
\end{proof}
 \remark Theorem \ref{p3theom14d} shows that the local-univalence assumption of the function $f_1*f_2$ in Theorem \ref{p3theom10} can be removed, if $g_2\equiv 0$. That is, if the function $f_2\in\mathcal{K}$.\

 \remark By taking the function  $\phi=z/(1-z)$ in  Theorem \ref{p3theom14d}, we see that a locally univalent and sense-preserving harmonic function  $f=h+\overline{g}$ is convex, if for any real numbers $\mu$ and $\nu$, it satisfies the equation \begin{equation}\label{p3eq14g}
 h(z)+\textit{e}^{-2\textit{i}\mu}g(z)=\int_0^z \frac{\textit{d}\xi}{1-2\xi\textit{e}^{\textit{i}\mu}\cos\nu+\xi^2\textit{e}^{2\textit{i}\mu}}.
 \end{equation}
 \remark Theorem \ref{p3theom14d} shows that the convolution of a locally univalent and sense-preserving harmonic mapping $f=h+\overline{g}$, satisfying $h+\textit{e}^{-2\textit{i}\mu}g=z/(1-z)$ (which obviously gives non-convex harmonic mappings), with the analytic function given in the $RHS$ of \eqref{p3eq14g} belongs to the class $\mathcal{K}_H^0$.
\section{Convolution of two harmonic mappings}
In this section, we prove univalence of the convolution $f_1*f_2$ for some specific harmonic functions $f_1$ and  $f_2$ determined by \eqref{p3eq11}.
Since the function $z/(1-\textit{e}^{\textit{i}\gamma} z)\in\mathcal{K}$ and is convolution identity, therefore, in view of Theorem \ref{p3theom10}, we have the following result.
\begin{theorem}\label{p3theom15}
Let the functions $f_1=h_1+\overline{g_1}$ and $f_2=h_2+\overline{g_2}$ be locally univalent and sense-preserving harmonic mappings in $\mathbb{D}$ such that, for some real numbers $\mu_1$, $\mu_2$ and $\nu$,
\[h_1(z)+\textit{e}^{-2\textit{i}\mu_1}g_1(z)=\frac{z}{1- z}\]and
 \[h_2(z)+\textit{e}^{-2\textit{i}\mu_2}g_2(z)=\int_0^z \frac{\textit{d}\xi}{1-2\xi\textit{e}^{\textit{i}(\mu_1+\mu_2)}\cos\nu+\xi^2\textit{e}^{2\textit{i}(\mu_1+\mu_2)}}.\]
Then the convolution $f_1*f_2$ is univalent and convex in the direction of $-(\mu_1+\mu_2)$, if it is locally univalent and sense-preserving.
\end{theorem}
The problem is now to check the local univalence and sense-preservity of $f_1*f_2$ in Theorem \ref{p3theom15}. In the next result, we check it for  the case where we fix $f_1$ to be the mapping defined in \eqref{p3eq04} and take the analytic dilatations of $f_2$ to be  $\omega(z)=az^n$ $(a\in\mathbb{C}: |a|\leq 1\quad\text{and}\quad n\in\mathbb{N})$. Lewy's theorem says it is enough to show  the analytic dilatation of  $f_1*f_2$ lies in $\mathbb{D}$. So, first we calculate the analytic dilatation of  $f_1*f_2$.
\begin{lemma}\label{p3lema16}
Let the function $f_1=h_1+\overline{g_1}$ be the harmonic right half-plane mapping defined in \eqref{p3eq04} and the function $f_2=h_2+\overline{g_2}$ be a locally univalent and sense-preserving harmonic mapping such
\begin{equation}\label{p3eq18}
h_2(z)+\textit{e}^{-2\textit{i}\mu}g_2(z)=\int_0^z \frac{\textit{d}\xi}{1-2\xi\textit{e}^{\textit{i}\mu}\cos\nu+\xi^2\textit{e}^{2\textit{i}\mu}}.
\end{equation}
If $\omega$ is the analytic dilatation of the function $f_2$, then the analytic dilatation of the convolution  $f_1*f_2$ is given by
\begin{equation}\label{p3eq19}
\omega_1=-z\frac{\omega'(1-2z\textit{e}^{\textit{i}\mu}\cos\nu+z^2\textit{e}^{2\textit{i}\mu})-\omega(1+\omega\textit{e}^{-2\textit{i}\mu})(-2z\textit{e}^{\textit{i}\mu}\cos\nu+2z\textit{e}^{2\textit{i}\mu})}{2(1+\omega\textit{e}^{-2\textit{i}\mu})(1-2z\textit{e}^{\textit{i}\mu}\cos\nu)-z\omega'\textit{e}^{-2\textit{i}\mu}(1-2z\textit{e}^{\textit{i}\mu}\cos\nu+z^2\textit{e}^{2\textit{i}\mu})}.
\end{equation}
\end{lemma}
\begin{proof}Since the function $f_1=h_1+\bar{g_1}$ is given by
\begin{equation}\label{p3eq17}
h_1(z)=\frac{1}{2}\left(\frac{z}{1-z}+\frac{z}{(1-z)^2}\right)\quad and \quad g_1(z)=\frac{1}{2}\left(\frac{z}{1-z}-\frac{z}{(1-z)^2}\right)
\end{equation}
therefore for any analytic function $F$,
\begin{equation}\label{p3eq20}
h_1*F=\frac{1}{2}\left(F+zF'\right)\quad and \quad g_1*F=\frac{1}{2}\left(F-zF'\right)
\end{equation}
Also, since $g_2'=\omega h_2'$, we have $g_2''=\omega h_2''+\omega' h_2'$. Therefore, in view of \eqref{p3eq20}, the analytic dilation $\omega_1$ of the convolution $f_1*f_2$ becomes
\begin{equation}\label{p3eq21}
\omega_1=\frac{(g_1*g_2)'}{(h_1*h_2)'}=-\frac{zg_2''}{2h_2'+zh_2''}=-\frac{z\omega'h_2'+z\omega h_2''}{2h_2'+zh_2''}
\end{equation}
Differentiating \eqref{p3eq18} and upon solving the resulting equation along with $g_2'=\omega h_2'$ for $h_2'$, we get
\begin{equation}\label{p3eq22}
h_2'=\frac{1}{(1+\textit{e}^{-2\textit{i}\mu}\omega)(1-2z\textit{e}^{\textit{i}\mu}\cos\nu+z^2\textit{e}^{2\textit{i}\mu})}.
\end{equation}
Again, on differentiating \eqref{p3eq22} gives
\begin{equation}\label{p3eq23}
h_2''=-\frac{(-2\textit{e}^{\textit{i}\mu}\cos\nu+z\textit{e}^{2\textit{i}\mu})(1+\omega \textit{e}^{-2\textit{i}\mu}+z \omega'\textit{e}^{-2\textit{i}\mu})+(1+\omega \textit{e}^{-2\textit{i}\mu})z\textit{e}^{2\textit{i}\mu}+\omega'\textit{e}^{-2\textit{i}\mu}}{(1+\textit{e}^{-2\textit{i}\mu}\omega)^2(1-2z\textit{e}^{\textit{i}\mu}\cos\nu+z^2\textit{e}^{2\textit{i}\mu})^2}.
\end{equation}
Using the expressions of $h_2'$ and $h_2''$ respectively from \eqref{p3eq22} and \eqref{p3eq23} in \eqref{p3eq21} and simplifying, we get \eqref{p3eq19}.
\end{proof}
\begin{theorem}\label{p3theom24}
Let the functions $f_1$ and $f_2$ be the harmonic mappings given in Lemma \ref{p3lema16}. and let $\omega=az^n$ $(|a|\leq1)$ be the analytic dilation of the function $f_2$. Then the convolution $f_1*f_2$ is univalent and convex in the direction of $-\mu$, if
\begin{enumerate}
\item  $n=1,2$ and $|a|\leq 1$, or
\item  $n\geq3$ and $|a|\leq n-1-\sqrt{n^2-2n}$.
\end{enumerate}
\end{theorem}
\begin{proof}
Note that
\[1+|a|^2-|a|(|2-n|+n)=\begin{cases}\left(1-|a|^2\right)^2, \quad \text{if}\quad n=1,2\\ \left(1+|a|^2-2|a|(n-1)\right), \quad \text{if}\quad n\geq 3.\end{cases}\]Therefore, by assumptions on $a$ and $n$, we get
\begin{equation}\label{p3eq24a}
\left(1+|a|^2-|a|(|2-n|+n)\right)\geq0
\end{equation}
Now, let $\omega_1$ be the  analytic dilation of the convolution $f_1*f_2$. To prove the result, Theorem \ref{p3theom15} shows that, it is sufficient to prove that $|\omega_1|<1$ on $\mathbb{D}$. Substituting $\omega=az^n$ in  \eqref{p3eq19}, on simplification we get
\begin{align}
\omega_1&=-z^n\left(\frac{a^2z^{n+2}-a^2\cos\nu\textit{e}^{-\textit{i}\mu}z^{n+1}+a(1-n/2)\textit{e}^{2\textit{i}\mu}z^2-a(1-n)\cos\nu\textit{e}^{\textit{i}\mu}z-an/2}{1-\cos\nu\textit{e}^{\textit{i}\mu}z+a(1-n/2)\textit{e}^{-2\textit{i}\mu}z^n-a(1-n)\cos\nu\textit{e}^{-\textit{i}\mu}z^{n+1}-anz^{n+2}/2}\right)\label{p3eq24}\\&=:-z^n\frac{p(z)}{q(z)},\notag
\end{align}
where $p$ and $q$ are respectively the numerator and denominator of the fraction in brackets of the above expression of $\omega_1$.
We have
\begin{align}
p(z)&=a^2z^{n+2}-a^2\cos\nu\textit{e}^{-\textit{i}\mu}z^{n+1}+a(1-n/2)\textit{e}^{2\textit{i}\mu}z^2-a(1-n)\cos\nu\textit{e}^{\textit{i}\mu}z-an/2\label{p3eq26}\\&=a^2(z-\textit{e}^{-\textit{i}\mu}\cos\nu)z^{n+1}+a((1-n/2)\textit{e}^{2\textit{i}\mu}z^2-(1-n)\cos\nu\textit{e}^{\textit{i}\mu}z-n/2)\notag
\end{align}
and
\begin{align}
q(z)&=(1-\textit{e}^{\textit{i}\mu}\cos\nu z)+a((1-n/2)\textit{e}^{-2\textit{i}\mu}z^n-(1-n)\cos\nu\textit{e}^{-\textit{i}\mu}z^{n+1}-nz^{n+2}/2)\label{p3eq27}\\&=(1-\textit{e}^{\textit{i}\mu}\cos\nu z)\left(1+a\left(1-n/2\right)\textit{e}^{-2\textit{i}\mu}z^n\right)+an\textit{e}^{-\textit{i}\mu}(\cos\nu-\textit{e}^{\textit{i}\mu}z)z^{n+1}/2\notag
\end{align}
Using the above form of $q$ and taking \[A=1-\textit{e}^{\textit{i}\mu}\cos\nu z\quad\text{and}\quad B=(1-n/2)\textit{e}^{-2\textit{i}\mu}z^n-(1-n)\cos\nu\textit{e}^{-\textit{i}\mu}z^{n+1}-nz^{n+2}/2,\] we get
\begin{align}
2\RE\left(q(z)/A\right)&=2+2\RE\left(aB/A\right)\notag\\&=2+\RE\left( a\left(2-n\right)\textit{e}^{-2\textit{i}\mu}z^n\right)+\RE\left(\frac{\cos\nu-\textit{e}^{\textit{i}\mu}z}{1-\textit{e}^{\textit{i}\mu}\cos\nu z}an\textit{e}^{-\textit{i}\mu}z^{n+1}\right)\notag\\&>2-|a|(|2-n|+n)>0\notag
\end{align}
by using \eqref{p3eq24a}. Hence, $q(z)\neq0$ on $\mathbb{D}$. Therefore the function $\omega_1$ is analytic on $\mathbb{D}$ for the said values of $a$ and $n$. Thus to show $|\omega_1|<1$ on $\mathbb{D}$, it is enough to show it on $|z|=1$.\

Again, using the above expressions given by \eqref{p3eq26} and \eqref{p3eq27} of $p$ and $q$, we see on $|z|=1$,
\begin{align*}
|q(z)|^2-|p(z)|^2&=|A+aB|^2\notag-|\bar{a}^2A+\bar{a}B|^2\\&=(1-|a|^2)\left((1+|a|^2)|A|^2+2\RE(aB\overline{A})\right)\notag\\&=(1-|a|^2)|A|^2\left(1+|a|^2+2\RE\left(aB/A\right)\right)\\&>(1-|a|^2)|A|^2\left(1+|a|^2-|a|(|2-n|+n)\right).
\end{align*}
Therefore, by using  \eqref{p3eq24a},  $(|q(z)|^2-|p(z)|^2)>0$ on $|z|=1$.
Hence, $|\omega_1|<1$ on $\mathbb{D}$.
\end{proof}
\remark Let the function $f_2=h_2+\bar{g_2}\in\mathcal{S}^0(H_{\mu})$ is a slanted half-plane mapping. Then, we have
\[h_2(z)+\textit{e}^{-2\textit{i}\mu}g_2(z)=\frac{z}{1-\textit{e}^{\textit{i}\mu}z}=\int_0^z \frac{\textit{d}\xi}{1-2\xi\textit{e}^{\textit{i}\mu}+\xi^2\textit{e}^{2\textit{i}\mu}}\] Now, let $\omega=az^n$ $(|a|\leq1)$  be the  analytic dilation of $f_2$, such that either
\begin{enumerate}
\item  $n=1,2$ and $|a|\leq 1$, or
\item  $n\geq3$ and $|a|\leq n-1-\sqrt{n^2-2n}$.
\end{enumerate} Then, by Theorem \ref{p3theom24}, the convolution $f_1*f_2\in\mathcal{S}_H^0$ and and is convex in the direction of $-\mu$. This proves Theorem \ref{p3lema010}. Similarly we can prove  Theorem \ref{p3lema011}.
\remark\label{p3remak28a} Result in Theorem \ref{p3theom24} is not true for $n\geq3$ when $|a|=1$. In this case the convolution $f_1*f_2$ is not necessarily locally univalent. To show this, it is enough to show that the analytic dilation $\omega_1$ of $f_1*f_2$ satisfies $|\omega_1(z)|>1$ for some $z\in\mathbb{D}$. Taking $a=\textit{e}^{\textit{i}\varphi}$ in \eqref{p3eq24}, we get\begin{align}
\omega_1&=-z^n\textit{e}^{2\textit{i}\varphi}\left(\frac{z^{n+2}-\cos\nu\textit{e}^{-\textit{i}\mu}z^{n+1}+\textit{e}^{-\textit{i}\varphi}\left((1-n/2)\textit{e}^{2\textit{i}\mu}z^2-(1-n)\cos\nu\textit{e}^{\textit{i}\mu}z-n/2\right)}{1-\cos\nu\textit{e}^{\textit{i}\mu}z+\textit{e}^{\textit{i}\varphi}\left((1-n/2)\textit{e}^{-2\textit{i}\mu}z^n-(1-n)\cos\nu\textit{e}^{-\textit{i}\mu}z^{n+1}-nz^{n+2}/2\right)}\right)\notag\\&=:-z^n\textit{e}^{2\textit{i}\varphi}\frac{p(z)}{\overline{z^{n+2}p(1/\overline{z})}},\label{p3eq28b}
\end{align}
where $p$ is the numerator  of the fraction in brackets of the above expression of $\omega_1$. Now, if $|\omega_1(z)|<1$ $(z\in\mathbb{D})$, \eqref{p3eq28b} implies $p$ has all the zeros in $\mathbb{D}$, whereas modulus of their product is $n/2$ and $n\geq3$. Therefore we have arrived at a contradiction. Hence, $|\omega_1(z)|>1$ for some $z\in\mathbb{D}$. \

In the next result, we don't fix the harmonic function $f_1$ to be the right half-plane mapping given by \eqref{p3eq04}, instead we take a class of directional convex harmonic mappings. Let the function $f_1=h_1+\overline{g_1}$ be a locally univalent and sense-preserving  harmonic mapping such that, for some real number $\mu_1$, $h_1+\textit{e}^{-2\textit{i}\mu_1}g_1=z/(1-z)$ and $g_1'=-\textit{e}^{\textit{i}\mu_1}zh'_1$. Then by the method of shear construction, we get
\begin{equation}\label{p3eq29}
h_1(z)=\frac{1}{2}\left(\frac{z}{1-z}+\frac{z}{(1-z)^2}\right)\quad\text{and}\quad g_1(z)=\frac{\textit{e}^{2\textit{i}\mu_1}}{2}\left(\frac{z}{1-z}-\frac{z}{(1-z)^2}\right)
\end{equation}
Looking at this mapping $f_1$ and the proof of Lemma \ref{p3lema16} we get easily the following extension of Theorem \ref{p3theom24}.
\begin{theorem}\label{p3theom30}
Let the function $f_1$ be the harmonic mapping given by \eqref{p3eq29}. Also, let the function $f_2=h_2+\overline{ g_2}$ be a locally univalent and sense-preserving harmonic mapping such that
\begin{equation}\label{p3eq31}
h_2(z)+\textit{e}^{-2\textit{i}\mu_2}g_2(z)=\int_0^z \frac{\textit{d}\xi}{1-2\xi\textit{e}^{\textit{i}(\mu_1+\mu_2)}\cos\nu+\xi^2\textit{e}^{2\textit{i}(\mu_1+\mu_2)}}.
\end{equation}
 Then the convolution $f_1*f_2$ is univalent and convex in the direction of $-(\mu_1+\mu_2)$, if
\begin{enumerate}
\item  $n=1,2$ and $|a|\leq 1$, or
\item  $n\geq3$ and $|a|\leq n-1-\sqrt{n^2-2n}$.
\end{enumerate}
\end{theorem}
\begin{proof}
From \eqref{p3eq29} and \eqref{p3eq31}, we see\[(h_2+\textit{e}^{-2\textit{i}\mu_1}g_2)(z)*(h_2+\textit{e}^{-2\textit{i}\mu_2}g_2)(z)=\int_0^z \frac{\textit{d}\xi}{1-2\xi\textit{e}^{\textit{i}(\mu_1+\mu_2)}\cos\nu+\xi^2\textit{e}^{2\textit{i}(\mu_1+\mu_2)}}.\]Therefore, by Theorem \ref{p3theom15}, it is enough to prove $f_1*f_2$ is locally univalent and sense-preserving. Let $\omega$ and $\omega_1$ be the analytic dilations respectively of the mappings $f_2$ and $f_1*f_2$. Clearly the function $\omega_2:=\textit{e}^{2\textit{i}\mu_1}\omega$ is the analytic dilation of the mapping $h_2+\overline{G_2}$, where $G_2=\textit{e}^{2\textit{i}\mu_1}g_2$. Putting $\mu=\mu_1+\mu_2$, \eqref{p3eq31} gives
\begin{equation}\label{p3eq32}
h_2(z)+\textit{e}^{-2\textit{i}\mu}G_2(z)=\int_0^z \frac{\textit{d}\xi}{1-2\xi\textit{e}^{\textit{i}\mu}\cos\nu+\xi^2\textit{e}^{2\textit{i}\mu}}\end{equation} As in the proof of Lemma \ref{p3lema16}, we find out that the analytic dilation  $\omega_1$ of $f_1*f_2$ is given by
\begin{equation}\label{p3eq33}
\omega_1=\frac{(g_1*g_2)'}{(h_1*h_2)'}=-\frac{z\textit{e}^{2\textit{i}\mu_1}g_2''}{2h_2'+zh_2''}=-\frac{zG_2''}{2h_2'+zh_2''}\end{equation}Equations \eqref{p3eq32} and \eqref{p3eq33} are respectively identical with \eqref{p3eq18} and \eqref{p3eq21}. Therefore,  Lemma \ref{p3lema16} shows that\begin{align*}
\omega_1=-z\frac{\omega_2'(1-2z\textit{e}^{\textit{i}\mu}\cos\nu+z^2\textit{e}^{2\textit{i}\mu})-\omega_2(1+\omega_2\textit{e}^{-2\textit{i}\mu})(-2z\textit{e}^{\textit{i}\mu}\cos\nu+2z\textit{e}^{2\textit{i}\mu})}{2(1+\omega_2\textit{e}^{-2\textit{i}\mu})(1-2z\textit{e}^{\textit{i}\mu}\cos\nu)-z\omega_2'\textit{e}^{-2\textit{i}\mu}(1-2z\textit{e}^{\textit{i}\mu}\cos\nu+z^2\textit{e}^{2\textit{i}\mu})}.
\end{align*}Above equation is same as \ref{p3eq19}, except $\omega$ replaced by $\omega_2$. Hence, the result follows by Theorem \ref{p3theom24}.
\end{proof}
\remark Similarly as in the  Remark \ref{p3remak28a}, we can show that Theorem \ref{p3theom30} does not hold for $n\geq3$, when $|a|=1$.
\problem Does there exists a harmonic function $f=h+\overline{g}\in\mathcal{S}_H^0$, satisfying that the function $h+\textit{e}^{\textit{i}\mu}g\notin\mathcal{K}$ for any of the real numbers $\mu$, and an analytic function $\phi\in\mathcal{K}$ such that the convolution $f*\phi\in\mathcal{K}_H^0$?

\end{document}